\documentclass[12pt, reqno]{amsart}
\usepackage{amsmath}
\usepackage{amsthm}
\usepackage{amssymb}
\usepackage{enumerate}
\usepackage{graphicx}
\usepackage{graphicx}
\usepackage{amssymb,amsfonts}
\usepackage{amsmath}
\usepackage[colorlinks,linkcolor=blue,anchorcolor=blue,citecolor=blue]{hyperref}
\usepackage[numbers,sort&compress]{natbib}
\usepackage{marvosym}
\usepackage{epstopdf}
\usepackage{caption}
\usepackage{multicol}
\usepackage{multirow}
\usepackage{makebox}
\usepackage{subfigure}
\numberwithin{equation}{section}
\newtheorem{theo}{Theorem} 

\newtheorem{lem}{Lemma}
\newtheorem{cor}{Corollary}

\begin{document}

\title{The log-Minkowski inequality of curvature entropy for non-symmetric convex bodies}

\author[Zeng, Dong, Wang and Ma]{Chunna Zeng$^1$, Xu Dong$^1$, Yaling Wang$^1$, Lei Ma$^{2*}$}

\address{1.School of Mathematical Sciences,
 Chongqing Normal University,
Chongqing 401331, People's Republic of China}
\email{zengchn@163.com}

\address{1.School of Mathematical Sciences,
 Chongqing Normal University,
Chongqing 401331, People's Republic of China}
\email{dx13370724790@163.com}

\address{1.School of Mathematical  Sciences,
 Chongqing Normal University,
Chongqing 401331, People's Republic of China}
\email{wangyl7228@163.com}

\address{2.School of Sciences, Guangdong Preschool Normal College in Maoming,
Maoming 525200,  People's Republic of China}
\email{maleiyou@163.com}

\thanks{The first author   was supported in part by the Major Special Project of NSFC (Grant No. 12141101), the Young Top-Talent program of Chongqing(Grant No. CQYC2021059145), NSF-CQCSTC(Grant No. cstc2020jcyj-msxmX0609), Technology Research Foundation of Chongqing Educational committee(Grant No. KJQN201900530, KJZD-K202200509), the Research Project of Chongqing Education Commission CXQT21014.}
\thanks{The second author was supported  by  the Characteristic innovation projects of universities in Guangdong province (Grant No. 2020KTSCX358).}
\thanks{{*}Corresponding author: Lei Ma}

\begin{abstract}
In an earlier paper \cite{mazeng} the authors introduced the notion of curvature entropy, and proved
 the plane log-Minkowski inequality of curvature entropy under the symmetry assumption. In this paper we demonstrate the plane log-Minkowski inequality of curvature entropy for general convex bodies. The equivalence of the uniqueness of cone-volume measure, the log-Minkowski inequality of volume, and the log-Minkowski inequality of curvature entropy for  general convex bodies in $\mathbb R^{2}$  are shown.
\end{abstract}

\maketitle

\section{Introduction }

The Brunn-Minkowski theory of convex bodies, also known as the Minkowski mixed volume theory, is the center of  convex geometric analysis. In fact the Brunn-Minkowski theory is the combination of the vector addition and volume. This point is fully demonstrated in Minkowski's paper \cite{H.Min}.
Every development of this theory just reflects this point:  once the addition of convex bodies is extended,  the Brunn-Minkowski theory will evolve into a new stage. The theory has experienced the classical stage, the $L_p$ stage, and the Orlicz stage. Over a hundred years, the Brunn-Minkowski theory has established important connections among many analytical branches in mathematics, such as probability and statistic, information theory, physics, ellipsoidal partial differential equations, algebraic geometry and so on (see \cite{R. Gardner} for more detail).

The Brunn-Minkowski inequality is a fundamental geometric inequality in the classical Brunn-Minkowski theory. It gives the log-concavity of the Lebesgue measure and implies the
classical isoperimetric inequality in the Euclidean space. Gardner's excellent survey
article \cite{R. Gardner}  describes its generalizations, consequences, and applications in geometry, analysis, probability, and other subjects.
During the last two decades, motivated by the study of geometry of $L_{p}$ spaces, the
$L_{p}$ Brunn-Minkowski theory has achieved enormous success when $p>1$. The classical
Brunn-Minkowski theory corresponds to the case of $p=1$. The case of $p<1$, in
particular, the singular case $p=0$, remains a challenge.

One of the most important problems in the $L_{p}$ Brunn-Minkowski theory is to establish
the corresponding Brunn-Minkowski inequality when $p<1$. In fact, the singular case
$p=0$ is strongest and implies all cases of $p>0$. Following the breakthrough of the work
\cite{{BLYZ2}}, the authors of \cite{BLYZ2} proved in \cite{BLYZ} the Brunn-Minkowski inequality for the case $p=0$ in
the plane which is called the log-Brunn-Minkowski inequality. The higher dimensional
case is still open and is considered as a major problem in convex geometry.

\textbf{The logarithmic Minkowski inequality of volume}. In \cite{BLYZ}, B\"or\"oczky, Lutwak, Yang and Zhang  conjectured that for origin-symmetric convex bodies $K$, $L$ in  $\mathbb{R}^n$ and $\lambda\in (0,1)$, then
\begin{align}\label{aa98}
V((1-\lambda) K+_0\lambda L)\ge V(K)^{1-\lambda} V(L)^ \lambda.
\end{align}
 (\ref{aa98}) is called log-Brunn-Minkowski inequality and note that it  is not true for general convex bodies. The log-Brunn-Minkowski inequality  is stronger than its classical counterpart  for origin-symmetric convex bodies. B\"or\"oczky, Lutwak, Yang and Zhang also demonstrated that     (\ref{aa98}) is equivalent to the following log-Minkowski inequality in $\mathbb{R}^n$
\begin{align}\label{aa97}
\int_{\mathbb{S}^{n-1}}\emph{\emph{log}} \left(\frac{h_L}{h_K} \right)dV_K \ge \frac{V(K)}{n}\emph{\emph{log}}\left(\frac{V(L)}{V(K)} \right),
\end{align}
where $V_K$ is the cone volume measure of  $K$. Furthermore, B\"or\"oczky, Lutwak, Yang and Zhang proved (\ref{aa97}) in planar case by using the uniqueness of cone-volume measures.

So far, the  log-Minkowski inequality and its extensions have generated a considerable literature. For instance, Rotem \cite{Rotem} solved   the log-Minkowski inequality in the complex space. Saroglou \cite{sar} verified the conjecture of (\ref{aa97}) when $K$ and $L$ are both simultaneously unconditional with respect to the same orthogonal basis, meaning that they are invariant under reflections with respect to the principle coordinate hyperplanes $x_i=0$. Colesanti, Livshyts and Marsiglietti \cite{CLM} verified (\ref{aa97}) locally for small-enough $C^2$-perturbations of the Euclidean ball $B$.  Xi and Leng \cite{2} proved Dar's conjecture and demonstrated the equivalence between the log-Brunn-Minkowski inequality and the log-Minkowski inequality for non-symmetric convex bodies.  Tao Xiong and Xiong \cite{TX} proved the log-Minkowski inequality for cylinders in $\mathbb R^{3}.$ Ma, Zeng and Wang \cite{mazeng} showed that the log-Minkowski inequality of volume is equivalent to the log-Minkowski inequality of curvature entropy in $\mathbb R^{n}.$

 \textbf{The $L_p$ Minkowski problem}. The $L_p$ Minkowski problem is one of the central problems in contemporary convex geometric analysis. The classical Minkowski problem states that: given a finite Borel measure $\mu$ on  $\mathbb{S}^{n-1}$,  what are the necessary and sufficient conditions so that $\mu$ is the surface area measure of a convex body $K$? Minkowski \cite{H.Min2} solved  this question when the given measure is either discrete or has a continuous density. Later, Aleksandrov \cite{AD2,AD3}, Fenchel and Jessen \cite{Fenchel} solved the problem for general measure.       It showed that if $\mu$ is not concentrated on any closed hemispherical surface, then $\mu$ is the surface area measure of $K$ when and only when its centroid is at the origin. The $L_p$ Minkowski problem is an extension of the classical   Minkowski problem and has achieved great developments.  The  $L_p$ Minkowski problem in the plane is solved by Stancu \cite{Stancu1,Stancu2}, Umanskiy \cite{Umanskiy}, Chen \cite{Chen}, and Jiang \cite{Jiang}. The solution of $L_p$ Minkowski problem is homothetic solutions of Gauss curvature flow, see \cite{Andrews1,Andrews2,Chow,Gage,Tso}. When $\mu$ is Lebesgue measure on unit circle $\mathbb{S}^1$, then the solution of $L_p$ Minkowski problem in $\mathbb{R}^2$ is homothetic solutions of misdirected curve flow, see Andrews \cite{Andrews2}. Obviously, the $L_1$ Minkouski problem is the classical Minkowski problem, while the $L_0$ Minkowski problem is the logarithmic Minkowski problem.

\textbf{The conjectured uniqueness of log-Minkowski problem.}  The log-Minkowski problem  is the most important case because  the   cone-volume measure is the only  $SL(n)$  invariant measure among all the $L_p$ surface area measure.

[\emph{Lutwak}] If $K$ and $L$ are symmetric smooth strictly convex sets with $V_K(\omega)=V_L(\omega)$, where $\omega \in \mathbb{S}^{n-1}$ is a Borel set, then $K=L$.

 Firey \cite{W.Firey2} proved that  if the cone-volume measure of a origin-symmetric convex body is a positive constant multiplied by spherical Lebesgue measure in $\mathbb{R}^n$, then the body must be an Euclidean ball. In \cite{BLYZ2}, B\"or\"oczky, Lutwak, Yang and Zhang showed that  if $K$, $L$ are origin-symmetric and have the same cone-volume measures in $\mathbb{R}^2$, then they are either parallelograms with parallel sides or $K=L$. The special case of smooth origin-symmetric planar convex bodies with positive curvature was proved
  by Gage \cite{Gage2}.     A nature question is whether the uniqueness  of  cone-volume measure holds without symmetrical condition.  In \cite{2}, Xi and Leng gave the definition of dilation position for the first time to prove the log-Brunn-Minkowski inequality for two convex bodies $K, L \in \mathcal{K}_{0}^{2}$     and solved the planar Dar's conjecture.  Zhu \cite{Zhu} and   Stancu \cite{Stancu1} solved  the case of discrete measure. However, the uniqueness condition is still remain open in higher dimensional case.

\textbf{The log-Minkowski inequality of curvature entropy for symmetric convex bodies.}
There are many entropy inequalities which have deep relationships with the Brunn-Minkowski inequality, such as Cramer-Rao inequalities, Fisher information inequality,    moment-entropy inequality, entropy power inequality, Stams inequality and so on, see \cite{Lutwak1,Lutwak2,Max}. For example, the entropy power inequality states that the effective variance (entropy power) of the sum of two independent random variables is greater than the sum of their effective variances. While the Brunn-Minkowski inequality states that the effective radius of the set sum of two sets is greater than the sum of their effective radii. Both these inequalities are recast in a form that enhances their similarity.

In  \cite{mazeng}, the authors introduced the notion of curvature entropy in $\mathbb{R}^n$. Assume that $K,~L \in \mathcal{K}^{n}_{0}$, then the curvature entropy  $E\left(K, L\right)$ is defined as
\begin{equation}\label{E(K,L)=-}
E(K,L)=-\int_{\mathbb{S}^{n-1}} \log \frac{H_{n-1}\left(L\right)}{H_{n-1}\left(K\right)} dV_K,
\end{equation}
where $H_{n-1}\left(\cdot\right)$ denotes the Gauss curvature of the boundary of a convex body. They obtained  the plane log-Minkowski inequality of curvature entropy in $\mathbb{R}^2$ when one convex body is symmetric. And they also demonstrated that under symmetry condition,  the uniqueness  of cone-volume measure, the log-Minkowski inequality of volume and the  log-Minkowski inequality of curvature entropy  are equivalent in $\mathbb{R}^n$. They conjectured the following problem

\textbf{Problem 1.} Suppose that $K, L\in \mathcal{K}^{n}_{0}$ are  strictly convex bodies  of $C^{2}$ boundaries  with $K$ symmetric in $\mathbb{R}^{n}$,
then
\begin{equation}\label{34}
E(K, L)\leq\frac{n-1}{n} V(K)\log \frac{ V(L)}{ V(K)},
\end{equation}
with equality holds when and only when $K$ and $L$ are homothetic.

By using techniques offered by \cite{2} and \cite{1}, the authors obtain the following plane log-Minkowski inequality of curvature entropy for general convex bodies without the symmetry assumption.

\begin{theo}\label{the1}
 Let $K, L \in  \mathcal{K}_{0}^{2}$  be two smooth,  planar  strictly convex bodies. If $K$ and $L$ are at a dilation position, then
\begin{equation}\label{fc16}
E(K,L)\leq\frac{V(K)}{2}\log\frac{ V(L)}{ V(K)},
\end{equation}
the equality holds when and only when $K$ and $L$ are homothetic.
\end{theo}

As an application of  above theorem, we obtain the following uniqueness condition for planar cone-volume measure of   non-symmetric convex bodies.

\begin{theo}\label{the2}
 Let $K, L \in  \mathcal{K}_{0}^{2}$ be two smooth, planar  strictly convex bodies. If $K$ and $L$ are at a dilation position and have the same cone-volume measure, then $K=L.$
\end{theo}

One of main aims of this paper is to show that for two general convex bodies in $\mathbb{R}^{2}$, the log-Minkowski inequality of volume and the log-Minkowski inequality of curvature entropy (\ref{fc16}) are equivalent. Furthermore, the equivalence of  the uniqueness  of cone-volume measure, the log-Minkowski inequality of volume and the  log-Minkowski inequality of curvature entropy  are established in $\mathbb{R}^2$.

\begin{theo}\label{the3}
Let $K, L \in  \mathcal{K}_{0}^{2}$ be two smooth, planar  strictly convex bodies. If $K$ and $L$ are at a dilation position, then the uniqueness  for cone-volume measures, the log-Minkowski inequality of volume and the  log-Minkowski inequality of curvature entropy are equivalent.
\end{theo}

%
%



\section{~Preliminaries}

In this section, we list  some notation and   basic facts about convex bodies. Good general references for the theory of  convex bodies see, e.g., \cite{Gardner,Gruber,Schneider}.

Write $x\cdot y$ for the standard inner product of $x,y \in \mathbb{R}^n$. A \emph{convex body} in $\mathbb R^{n}$ is a compact convex subset  with non-empty interior  in $\mathbb{R}^n$. Denote by $\mathcal{K}^{n}$ and $\mathcal{K}^{n}_{0}$ the set of convex bodies and convex bodies containing the origin in their interior in $\mathbb{R}^n$.
A convex body $K$  is uniquely determined by its support function $h_K: \mathbb{R}^n \rightarrow \mathbb{R}$,
\begin{equation*}
h_K(x)=\max\left\{x\cdot y:y \in K\right\}.
\end{equation*}
The support function is positively homogeneous of degree one and subadditive.

 Assume that $K$ is a strictly convex body of $C^2$ boundary in $\mathbb{R}^n$. Let $u: \partial K\rightarrow \mathbb{S}^{n-1}$ be the Gauss map. Then the surface area element $dS(x)$ of $\partial K$ and the surface area element of $\mathbb{S}^{n-1}$ at $u$ are related by

\begin{equation}\label{dSX}
dS(x)=\frac{1}{H_{n-1}}du,
\end{equation}
where $H_{n-1}$ is the Gauss curvature of $\partial K$ at $x$. By the following expression of the reciprocal Gauss curvature (\cite{Zhang})

\begin{equation*}
\frac{1}{H_{n-1}}=\det(h_{ij}+h_K\delta_{ij}),
\end{equation*}
where $h_{ij}$ is the covariant derivative of $h$ with respect to an orthonormal frame on $\mathbb{S}^{n-1}$ and $\delta_{ij}$ is the Kronecker delta,
 (\ref{dSX}) can be rewritten as
\begin{equation*}
dS(x)=\det(h_{ij}+h_K\delta_{ij})du.
\end{equation*}
Thus

\begin{equation*}
V\left(K\right)=\frac{1}{n}\int_{\mathbb{S}^{n-1}}h_K\det(h_{ij}+h_K\delta_{ij})du.
\end{equation*}

The relative curvature radius of $K$ with respect to $L$ in $\mathbb{R}^n$ is the ratio of Gauss curvature of $K$ and $L$, which is defined by

\begin{equation}\label{B1}
\rho_{K,L}=\frac{H_{n-1}\left(L\right)}{H_{n-1}\left(K\right)}=\frac{det\left(h_{ij}\left(K\right)+h_K\delta_{ij}\right)}{det\left(h_{ij}\left(L\right)+h_L\delta_{ij}\right)}.
\end{equation}
Specially, in the plane case,
\begin{equation}\label{B2}
\rho_{K,L}=\frac{\kappa_L}{\kappa_K}=\frac{h_K+h_K^{\prime\prime}}{h_L+h_L^{\prime\prime}},
\end{equation}
where $\kappa$ denotes the curvature of closed convex curve. If $K$ and $L$ are homothetic, then $\rho_{K,L}$ is a constant.

 For  $K, L \in\mathcal{K}^{n}$, the Minkowski addition is
defined by
$K+L=\{x+y:x \in K, y\in L\},$
and the multiple of ~$K$ by the scalar $\lambda> 0$ is
$\lambda K=\{\lambda x: x\in K\}.$
The Minkowski-Steiner formula states that
 \begin{align}\label{MS}
V(K+\lambda L)=\sum\limits _{i=0}^{n} \left( \begin{aligned}
        &n \\
           &i\end{aligned} \right)V_{i}(K, L)\lambda^{i},
\end{align}
where $\lambda > 0.$ $V(\cdot)$ is the $n$-dimensional volume and the
coefficients
\begin{equation*}
V_{i}(K, L)=V(\underbrace{K, \cdots, K}_{n-i}
\underbrace{L,\cdots,L}_{i})
\end{equation*}
are called \emph{ mixed volumes} of $K$ and $L,$ which are
nonnegative, symmetric in the indices, multilinear, translation invariant and homogeneous.

Actually, (\ref{MS}) can be rewritten as the following form, which gives the relative Steiner polynomial of $K$ with respect to $L$, that is
\begin{equation}\label{mix volume2}
V\left(K+t L\right)=V(K)+2V(K,L)t+V(L)t^2.
\end{equation}
From the Minkowski inequality
\begin{equation*}
V^{2}(K,L)-V(K)V(L)\ge0,
\end{equation*}
it follows that the expression $V(K+tL)=0$ has two negative real roots. Denote by $t_1$ and $t_2$ ($t_1>t_2$) the two roots of the relative Steiner polynomial of $K$ with respect to $L$
\begin{equation}\label{fc11}
t_1=-\frac{V(K,L)}{V(L)} +\frac{\sqrt{V(K,L)^{2}-V(K)V(L)}}{V(L)},
\end{equation}
\begin{equation}\label{fc12}
t_2=-\frac{V(K,L)}{V(L)} -\frac{\sqrt{V(K,L)^{2}-V(K)V(L)}}{V(L)}.
\end{equation}

The following Lemma introduced nonsymmetric extension of \emph{Green-Osher inequality}.
\begin{lem}\label{lem1}(\cite{1})
Let $K$ and $L$ be two smooth,  planar  strictly convex bodies and $\rho(\theta)$ the relative curvature radius of $K$ with respect to $L.$ If $K$ and $L$ are at a dilation position and ~$F(x)$ is a strictly convex function on ~$(0,+\infty),$ then
\begin{equation}\label{fc14}
\frac{1}{V(L)}\int^{2\pi}_{0} F(\rho(\theta))h_L(\theta)(h_L(\theta)+h_L''(\theta)) d\theta \geq F(-t_1)+F(-t_2),
\end{equation}
where ~$t_1,$~$t_2$ are the two roots of the relative Steiner polynomial of~$K$ with respect to~$L,$ and the equality  holds when and only when $K$ and $L$ are homothetic.
\end{lem}

 The cone-volume measure $V_K$ of $K\in\mathcal{K}^{n}_{0}$ is a Borel measure on the unit sphere $\mathbb{S}^{n-1}$ defined for a Borel set $\omega\subseteq \mathbb{S}^{n-1}$,  by
\begin{align*}
V_K\left(\omega\right)=\frac{1}{n}\int_{x\in v_K^{-1}\left(\omega\right)}{x\cdot v_k\left(x\right)d\mathcal{H}^{n-1}}\left(x\right),
\end{align*}
where $v_K: \partial K\rightarrow \mathbb{S}^{n-1}$ is the generalized Gauss map defined on $\partial K$, $\mathcal{H}^{n-1}$ is $(n-1)$-dimensional Hausdorff measure.
There are formulas
\begin{equation*}\label{a215}
dV_K=\frac{1}{n}h_KdS_K,
\end{equation*}
and
\begin{equation*}\label{a216}
V(K)=\frac{1}{n}\int_{\mathbb S^{n-1}}h_K(u)dS_K(u).
\end{equation*}
The mixed cone-volume measure $V_{K, L}$ of $K, L \in \mathcal{K}_{0}^{n}$ is introduced by Hu and Xiong \cite{HX}

\begin{equation*}
V_{K, L}\left(\omega\right)=\frac{1}{n}\int_{\omega}h_L(u)dS_K.
\end{equation*}
The total mass of $V_{K, L}$ is exactly the 1st mixed volume $V_{1}(K, L)$. In the Euclidean plane, it is clear that $V_{1}(K, L)=V_{1}(L, K)$, and we  write $V(K, L)$ instead of $V_{1}(K, L)$.

Let  $K$, $L \in \mathcal{K}^{n}$, $K$ and $L$ are at a \emph{dilation position} if the origin $o \in K\cap L$ and
\begin{equation*}
r(K, L)L\subset K \subset R(K, L)L.
\end{equation*}
where $r(K, L)$ and $R(K, L)$ are the inradius and outradius of $K$ with respect to $L$, i.e.
\begin{equation*}
r\left(K, L\right)=\max\{t>0:x+tL\subset K ~and~x\in \mathbb{R}^n\},
\end{equation*}
\begin{equation*}
R\left(K, L\right)=\min\{t>0:x+tL\supset K ~and~x\in \mathbb{R}^n\}.
\end{equation*}
The concept of dilation position is  based on solving the problem provided by G. Zhang in 2013, that is:~Let $K, L \in \mathcal{K}^{2}$, if there is  a ``suitable" position of the origin so that $K$ and  a ``suitable" translation of  $L$ satisfy
\begin{equation*}
V((1-\lambda)K +_0 \lambda L )\ge V(K)^{1-\lambda}V(L),
\end{equation*}
where $K$ and $L$ are origin-symmetric convex bodies in the plane.

%
%
%

\par The following is the general log-Minkowski inequality for planar convex bodies.
\begin{lem}\label{lem4}(\cite{2})
Let~$K,L\in \mathcal{K}^{2}$ with~$o\in K \cap L.$ If~$K$ and~$L$ are at a dilation position, then
\begin{equation}\label{fc15}
\int_{\mathbb{S}^{1}} \log \frac{h_L}{h_K} d V_K \geq \frac{V(K)}{2} \log \frac{V(L)}{V(K)}.
\end{equation}
 The equality holds when and only when~$K$ and~$L$ are dilates or~$K$ and~$L$ are parallelograms with parallel sides.
\end{lem}

\section{~ the log-Minkowski inequality of curvature entropy  with nonsymmetry  in $\mathbb R^{2}$}

%
%

%
%
%

\emph{Proof  of Theorem \ref{the1}.}
 Let ~$F(x)=-\log x$. By Lemma \ref{lem1}, (\ref{fc11}) and (\ref{fc12}), it follows that
\begin{equation*}
\begin{aligned}
\frac{1}{V(L)}\int^{2\pi}_{0} (-\log\frac{h_K+h_K^{\prime\prime}}{h_L+h_L^{\prime\prime}})h_L(\theta)(h_L(\theta)+h_L''(\theta)) d\theta
&\geq -\log(-t_1)-\log(-t_2)\\
&=-\log(t_1 t_2)\\
&=-\log\frac{V(K)}{V(L)}.
\end{aligned}
\end{equation*}
That is
\begin{equation*}
-\frac{2}{V(L)}\int_{\mathbb S^{1}}\log\frac{h_K+h_K^{\prime\prime}}{h_L+h_L^{\prime\prime}} d V_L \geq -\log\frac{V(K)}{V(L)}.
\end{equation*}
Thus
\begin{equation*}
\int_{\mathbb S^{1}} \log(\frac{h_K+h_K^{\prime\prime}}{h_L+h_L^{\prime\prime}}) d V_L \leq \frac{V(L)}{2}\log\frac{V(K)}{V(L)}.
\end{equation*}
By (\ref{E(K,L)=-}), (\ref{B1}) and (\ref{B2}) for planar case, we have
\begin{equation}\label{E(LK)}
E(L,K)\leq\frac{V(L)}{2}\log \frac{ V(K)}{ V(L)}.
\end{equation}
Note that $K$ and $L$ are at a dilation position, so
\begin{equation}\label{E(KL)}
E(K, L)\leq\frac{V(K)}{2}\log \frac{ V(L)}{ V(K)}.
\end{equation}
From the equality condition of (\ref{fc14}), it concludes that the equality holds in (\ref{fc16}) when and only when $K$ and $L$ are homothetic.
\qed

Especially, assume that $L=B$ is an unit Euclidean ball in $\mathbb{R}^2$, then

\begin{cor}\label{cor1}(\cite{wuff})
 Let $K \in \mathcal{K}_{0}^{2} $    be  smooth,  planar  strictly convex bodies, then
\begin{equation}\label{p}
\int_{\partial K} \kappa \log \kappa ds+ \pi \log\frac{V(K)}{\pi}\ge0,
\end{equation}
the equality holds when and only when $K$ is a ball.
\end{cor}
\begin{proof}
According to the definition of $E(L,K)$, it follows that

\begin{equation*}\begin{split}E(L,K)
&=-\int_{\mathbb S^{n-1}}\log\frac{H_{n-1}(K)}{H_{n-1}(L)}dV_{L}\\
&=-\frac{1}{n}\int_{\mathbb S^{n-1}}\log\frac{H_{n-1}(K)}{H_{n-1}(L)}h_{L}\det\big(h_{ij}(L)+h_{L}\delta_{ij}\big)du\\
&=-\frac{1}{n}\int_{\mathbb S^{n-1}}\log\frac{H_{n-1}(K)}{H_{n-1}(L)}h_{L}\frac{\det\big(h_{ij}(L)+h_{L}\delta_{ij}\big)}{\det\big(h_{ij}(K)+h_{K}\delta_{ij}\big)}\det\big(h_{ij}(K)+h_{K}\delta_{ij}\big)du\\
&=-\int_{\mathbb S^{n-1}}\frac{H_{n-1}(K)}{H_{n-1}(L)}\log\frac{H_{n-1}(K)}{H_{n-1}(L)}dV_{K, L}.
\end{split}
\end{equation*}
Taking $n=2$ and $L=B$, it yields
\begin{equation}\label{q1}
\begin{split}
E(B,K)
&=-\int_{\mathbb S^{1}}{\kappa}\log\kappa dV_{K, B}
\\&=-\frac{1}{2}\int_{\partial K}{\kappa}\log\kappa ds
\\&\le\frac{\pi}{2}\log \frac{ V(K)}{ \pi}.
\end{split}
\end{equation}
It follows that
\begin{equation*}
\int_{\partial K} \kappa \log \kappa ds+ \pi \log\frac{V(K)}{\pi}\ge0.
\end{equation*}
\end{proof}

\begin{cor}(\cite{wuff})
Let $K \in \mathcal{K}_{0}^{2}$  be  smooth,  planar  strictly convex bodies. Then
\begin{equation}
\int_{\partial K} \kappa \log \left(\kappa\sqrt\frac{V(K)}{\pi}\right) ds\ge0,
\end{equation}
the equality holds when and only when $K$ is a ball.
\end{cor}
\begin{proof}

Taking $ds=(h+h'')d\theta$, $\kappa=\frac{1}{h+h''}$  and $\int_{\partial K}d\theta =2\pi$ in (\ref{p}), we have
\begin{equation*}
\begin{aligned}
\int_{\partial K} \kappa \log \kappa ds+ \pi \log\frac{V(K)}{\pi}&=\int_{\partial K} \kappa \log \kappa ds+\frac{1}{2}\log\frac{V(K)}{\pi}\int_{\partial K}d\theta
\\&=\int_{\partial K} \kappa \log \kappa ds+\frac{1}{2}\int_{\partial K} \frac{1}{h+h''}(h+h'')\log\frac{V(K)}{\pi}d\theta
\\&=\int_{\partial K} \kappa \log \kappa ds+\frac{1}{2}\int_{\partial K} \kappa\log\frac{V(K)}{\pi}ds
\\&=\int_{\partial K} \kappa \log \left(\kappa\sqrt\frac{V(K)}{\pi}\right) ds
\\&\ge0.
\end{aligned}
\end{equation*}
\end{proof}

\section{The uniqueness  for the planar Cone-volume Measures  with nonsymmetry}

\emph{Proof of Theorem \ref{the2}.}
Assuming~$V_{K}=V_{L},$ it is obvious that~$V(K)=V(L).$ From (\ref{E(LK)}) we have
$$
E(L,K)\leq \frac{V(L)}{2} \log {\frac{V(K)}{V(L)}}=0,
$$
that is
\begin{equation}\label{fc17}
\int_{\mathbb{S}^{1}}\log (h_K +h''_K)d V_L \leq \int_{\mathbb{S}^{1}} \log(h_L +h''_L) d V_L.
\end{equation}
Similarly, by (\ref{E(KL)})
$$
E(K,L)\leq \frac{V(K)}{2} \log {\frac{V(L)}{V(K)}}=0,
$$
which is equivalent to
\begin{equation}\label{fc18}
\int_{\mathbb{S}^{1}}\log (h_L +h''_L)d V_K \leq \int_{\mathbb{S}^{1}} \log(h_K +h''_K) d V_K.
\end{equation}
By (\ref{fc17}), (\ref{fc18}) and~$V_K=V_L,$ then
\begin{equation*}
\begin{aligned}
\int_{\mathbb{S}^{1}}\log (h_L +h''_L)d V_K
&\leq \int_{\mathbb{S}^{1}} \log(h_K +h''_K) d V_K \\
&=\int_{\mathbb{S}^{1}} \log(h_K +h''_K) d V_L \\
&\leq \int_{\mathbb{S}^{1}} \log(h_L +h''_L) d V_L \\
&=\int_{\mathbb{S}^{1}} \log(h_L +h''_L) d V_K.
\end{aligned}
\end{equation*}
So
\begin{equation*}
\int_{\mathbb{S}^{1}} \log(h_L +h''_L) d V_K =\int_{\mathbb{S}^{1}} \log(h_K +h''_K) d V_K.
\end{equation*}
It yields
\begin{equation*}
E(K,L)=\frac{V(K)}{2} \log \frac{V(L)}{V(K)}=0.
\end{equation*}
By Theorem \ref{the1}, it follows that~$K$ and~$L$ are homothetic. Because of~$V(K)=V(L),$ we conclude that~$K=L.$
\qed

\section{Equivalence between uniqueness for cone-volume measures, log-Minkowski inequality of volume and log-Minkowski inequality of
 curvature entropy in $\mathbb R^{2}$ }

\emph{Proof of Theorem \ref{the3}.}
By Lemma \ref{lem4}, if $K, L\in \mathcal{K}^{2}_{0}$  are at a dilation position, it concludes that
 the uniqueness of the cone volume measure infers  the log-Minkowski inequality in $\mathbb R^{2}$. In order to prove that the uniqueness for cone-volume measures, the log-Minkowski inequality   of volume and the log-Minkowski inequality of curvature entropy are equivalent in $\mathbb R^{2}$,  we only need to prove that the log-Minkowski inequality of volume can deduce the log-Minkowski inequality of curvature entropy.

By the definition of $E(K,L)$, we have
\begin{equation}\label{nn}
\begin{aligned}
\int_{\mathbb{S}^{1}} \log\frac{h_L}{h_K}d V_K +E(K,L)&=\int_{\mathbb{S}^{1}} \log \frac{h_L(h_L+h''_L)}{h_K(h_K+h''_K)}d V_K
\\&\le V(K)\log\left(\frac{\int_{\mathbb{S}^{1}}\frac{h_{L}(h_{L}+h''_{L})}{h_{K}(h_{K}+h''_{K})}dV_{K}}{V(K)}\right)
\\&=V(K)\log\frac{V(L)}{V(K)},
\end{aligned}
\end{equation}
where the second step in (\ref{nn}) is due to Jensen's inequality. It follows that
\begin{equation*}
\int_{\mathbb{S}^{1}}\log\left(\frac{h_{L}}{h_{K}}\right) dV_{K}\leq V(K)\log\frac{V(L)}{V(K)}-E(K,L).
 \end{equation*}
By the log-Minkowski inequality (\ref{fc15}), we obtain
\begin{equation}\label{fc19}
\frac{V(K)}{2}\log\frac{V(L)}{V(K)} \leq\int_{\mathbb{S}^{1}}\log\left(\frac{h_{L}}{h_{K}}\right) dV_{K}\leq V(K)\log\frac{V(L)}{V(K)}-E(K,L).
\end{equation}
Then
\begin{equation*}\label{2231}
E(K,L)\leq\frac{V(K)}{2}\log \frac{ V(L)}{ V(K)}.
\end{equation*}

Next, we consider the equality condition of the log-Minkowski inequality for  curvature entropy.~Suppose that
\begin{equation*}
E(K,L)=\frac{V(K)}{2}\log \frac{ V(L)}{ V(K)}.
\end{equation*}
From (\ref{fc19}) it leads to
\begin{equation*}
\frac{V(K)}{2}\log\frac{V(L)}{V(K)}\leq\int_{\mathbb{S}^{1}}\log\left(\frac{h_{L}}{h_{K}}\right) dV_{K}\leq \frac{V(K)}{2}\log\frac{V(L)}{V(K)}.
\end{equation*}
i.e.
\begin{equation*}
\int_{\mathbb{S}^{1}}\log\left(\frac{h_{L}}{h_{K}}\right) dV_{K}=\frac{V(K)}{2}\log\frac{V(L)}{V(K)},
\end{equation*}
then $K$ and $L$ are dilates. In addition, if $K$ and $L$ are dilates, $E(K,L)=\frac{V(K)}{2}\log \frac{ V(L)}{ V(K)}$  obviously holds.
\qed

\section{~The entropy inequality  in $\mathbb R^{n}$ and some remarks}
In this section, an entropy inequality under a more general condition in $\mathbb{R}^n$ is established. The conditions of dilation position and symmetry are removed.

\emph{Proof of Theorem \ref{xx}.}
Note that
\begin{align}\label{fc1}
\int_{\mathbb S^{n-1}}\frac{H_{n-1}(L)}{H_{n-1}(K)}\log{\frac{H_{n-1}(L)}{H_{n-1}(K)}} dV_{L,K} =\int_{\mathbb S^{n-1}}\log{\frac{H_{n-1}(L)}{H_{n-1}(K)}}dV_K.
\end{align}
By Lebesgue's dominated convergence theorem,~as $p\rightarrow \infty,$
\begin{align*}
\int_{\mathbb S^{n-1}}\left(\frac{H_{n-1}(L)}{H_{n-1}(K)}\right)^{\frac{p}{p+n}}dV_{L,K} \rightarrow V(K),
\end{align*}
and
\begin{align*}
\int_{\mathbb S^{n-1}}\left(\frac{H_{n-1}(L)}{H_{n-1}(K)}\right)^{\frac{p}{p+n}}\log{\frac{H_{n-1}(L)}{H_{n-1}(K)}} dV_{L,K}\rightarrow \int_{\mathbb S^{n-1}}\left(\frac{H_{n-1}(L)}{H_{n-1}(K)}\right)\log{\frac{H_{n-1}(L)}{H_{n-1}(K)}} dV_{L,K}.
\end{align*}
Consider the function $f_{K,L}:[1,\infty]\rightarrow \mathbb R$ defined by
\begin{align*}
f_{K,L}(p)=\frac{1}{V(K)}\int_{\mathbb S^{n-1}}\left(\frac{H_{n-1}(L)}{H_{n-1}(K)}\right)^{\frac{p}{p+n}}dV_{L,K},
\end{align*}
and calculate, using L'H$\hat{o}$pital's rule,
$$
\begin{aligned}
\lim_{p\rightarrow \infty} \log(f_{K,L}(p))^{p+n}
&=\lim_{p\rightarrow \infty}\frac{\log{f_{K,L}(p)}}{\frac{1}{p+n}}\\
&=\lim_{p\rightarrow \infty}\frac{f^{'}_{K,L}(p)}{-\frac{f_{K,L}(p)}{(p+n)^2}}
\\&=\lim_{p\rightarrow \infty}{\frac{\frac{n}{(p+n)^{2}}\int_{\mathbb S^{n-1}}(\frac{H_{n-1}(L)}{H_{n-1}(K)})^{\frac{p}{p+n}}\log{\frac{H_{n-1}(L)}{H_{n-1}(K)}}dV_{L,K}}{\frac{-f_{K,L}(p)V(K)}{(p+n)^{2}}}}\\
&=-\frac{n}{V(K)}\int_{\mathbb S^{n-1}}\frac{H_{n-1}(L)}{H_{n-1}(K)}\log{\frac{H_{n-1}(L)}{H_{n-1}(K)}}dV_{L,K}.
\end{aligned}
$$
Consequently, we have

\begin{equation}\label{fc2}
\begin{aligned}
\exp\left[-\frac{n}{V(K)}\int_{\mathbb S^{n-1}}\frac{H_{n-1}(L)}{H_{n-1}(K)}\log{\frac{H_{n-1}(L)}{H_{n-1}(K)}}dV_{L,K}\right]\\
=\lim_{p\rightarrow \infty}\left[\frac{1}{V(K)}\int_{\mathbb S^{n-1}}\left(\frac{H_{n-1}(L)}{H_{n-1}(K)}\right)^{\frac{p}{p+n}}dV_{L,K}\right]^{p+n}.
\end{aligned}
\end{equation}
Via H$\ddot{o}$lder's inequality
\begin{align*}
\left[\int_{\mathbb S^{n-1}}\left(\frac{H_{n-1}(L)}{H_{n-1}(K)}\right)^{\frac{p}{p+n}}dV_{L,K}\right]^{\frac{p+n}{p}}\cdot \left[\int_{\mathbb S^{n-1}}dV_{L,K}\right]^{-\frac{n}{p}}\leq \int_{\mathbb S^{n-1}}\frac{H_{n-1}(L)}{H_{n-1}(K)}dV_{L,K}=V(K),
\end{align*}
which is equivalent to
\begin{align*}
\left[\frac{1}{V(K)}\int_{\mathbb S^{n-1}}\left(\frac{H_{n-1}(L)}{H_{n-1}(K)}\right)^{\frac{p}{p+n}}dV_{L,K}\right]^{\frac{p+n}{p}}\leq \frac{V(K){V_1(L,K)}^{\frac{n}{p}}}{{V(K)}^{\frac{p+n}{p}}}=\left[\frac{V_1(L,K)}{V(K)}\right]^{\frac{n}{p}}.
\end{align*}
It implies that
\begin{align*}
\lim_{p\rightarrow \infty}\left[\frac{1}{V(K)}\int_{\mathbb S^{n-1}}\left(\frac{H_{n-1}(L)}{H_{n-1}(K)}\right)^{\frac{p}{p+n}}dV_{L,K}\right]^{p+n}\leq \left[\frac{V_1(L,K)}{V(K)}\right]^{n}.
\end{align*}
By (\ref{fc2}),~we have
\begin{align*}
\exp\left[-\frac{n}{V(K)}\int_{\mathbb S^{n-1}}\frac{H_{n-1}(L)}{H_{n-1}(K)}\log{\frac{H_{n-1}(L)}{H_{n-1}(K)}}dV_{L,K}\right]\leq \left[\frac{V_1(L,K)}{V(K)}\right]^{n},
\end{align*}
it yields
\begin{equation*}
-\frac{n}{V(K)}\int_{\mathbb S^{n-1}}\frac{H_{n-1}(L)}{H_{n-1}(K)}\log{\frac{H_{n-1}(L)}{H_{n-1}(K)}}dV_{L,K}\le n\log\frac{V_1(L,K)}{V(K)},
\end{equation*}
that is
\begin{align*}
-\int_{\mathbb S^{n-1}}\frac{H_{n-1}(L)}{H_{n-1}(K)}\log{\frac{H_{n-1}(L)}{H_{n-1}(K)}}dV_{L,K}\leq V(K)\log\left(\frac{V_1(L,K)}{V(K)}\right).
\end{align*}
By (\ref{fc1})
\begin{align*}
-\int_{\mathbb S^{n-1}}\log{\frac{H_{n-1}(L)}{H_{n-1}(K)}}dV_K \leq V(K)\log\left(\frac{V_1(L,K)}{V(K)}\right),
\end{align*}
thus
\begin{align*}
E(K,L)\leq V(K)\log\frac{V_1(L,K)}{V(K)}.
\end{align*}
\qed

\end{document}